\documentclass[12pt, reqno]{amsart}
\usepackage{amsmath, amsthm, amscd, amsfonts, amssymb, mathrsfs, graphicx, color}
\usepackage[colorlinks=true, urlcolor=blue,  linkcolor=blue, citecolor=blue]{hyperref}
\input{mathrsfs.sty}

\newtheorem{theorem}{Theorem}[section]
\newtheorem{lemma}[theorem]{Lemma}
\newtheorem{proposition}[theorem]{Proposition}

\theoremstyle{definition}
\newtheorem{definition}[theorem]{Definition}
\newtheorem{example}[theorem]{Example}

\theoremstyle{remark}
\newtheorem{remark}[theorem]{Remark}
\begin{document}
\setcounter{page}{1}

\title[ Matricial Radius]{ Matricial Radius:\\
A Relation of Numerical radius with Matricial Range}
\author[  Mohsen
Kian,   Mahdi Dehghani  \MakeLowercase{and}  Mostafa Sattari ]{ M. Kian, M. Dehghani \MakeLowercase{and} M. Sattari}

\address{Mohsen Kian:\ \
 Department of Mathematics,  University of Bojnord, P. O. Box
1339, Bojnord 94531, Iran}
\email{\textcolor[rgb]{0.00,0.00,0.84}{kian@ub.ac.ir }}

\address{Mahdi Dehghani:\  Department of Pure Mathematics, Faculty of Mathematical Sciences,
University of Kashan, P. O. Box 87317-53153, Kashan, Iran}
\email{\textcolor[rgb]{0.00,0.00,0.84}{m.dehghani@kashanu.ac.ir \ and \ e.g.mahdi@gmail.com }}

\address{  Mostafa Sattari: \ Faculty of Basic Sciences, Department of Mathematics, University of Zabol, Zabol, Iran.}
\email{\textcolor[rgb]{0.00,0.00,0.84}{sattari@uoz.ac.ir \ and \ msattari.b@gmail.com }}

\subjclass[2010]{Primary 15A60 ; Secondary 47A12. }

\keywords{Numerical range, matricial range, completely positive mapping, numerical radius}

%%%%%%%%%%%%%%%%%%%%%%%%%%%%%%%%%%% abstract
%%%%%%%%%%%%%%%%%%%%%%%%%%%%%%%%%%%%%%%%%%%%%%%%%%%%%%%%%%%%%

\begin{abstract}
It has been shown that if $T$ is a complex matrix, then
{\small\begin{align*}
\omega(T)&=\frac{1}{n}\sup\left\{|\mathrm{Tr}\ X|;\ X\in W^n(T)\right\}\\
&=\frac{1}{n}\sup\left\{\|X\|_1;\ X\in W^n(T)\right\}\\
&= \sup\left\{ \omega(X);\ X\in W^n(T)\right\}
\end{align*} }
 where $n$ is a positive integer, $\omega(T)$ is the numerical radius and  $W^n(T)$ is the  $n$'th matricial range of $T$.
 \end{abstract} \maketitle

\maketitle
%%%%%%%%%%%%%%%%%%%%%%%%%%%%%%%%%%% abstract %%%%%%%%%%%%%%%%%%%%%%%%%%%%%%%%%%%%%%%%%%%%%%%%%%%%%%%%%%%%%

\section{Introduction and Preliminaries}

One of the most well-known concept in study of
Hilbert space operators   is the notion of numerical range. Assume that  $(\mathscr{H},\langle\cdot,\cdot\rangle)$ is a Hilbert space and   $\mathbb{B}(\mathscr{H})$ is the $C^*$-algebra of all bounded linear operators on $\mathscr{H}$ with the identity operator $I$. When $\mathscr{H}$ has finite dimension $n$, we identify $\mathbb{B}(\mathscr{H})$ with the algebra $\mathbb{M}_n:=\mathbb{M}_n(\mathbb{C})$ of all $n\times n$ complex matrices and $I_n$ denotes the $n\times n$ identity matrix. The numerical range  of   $T\in\mathbb{B}(\mathscr{H})$ is well-known:
 $$W(T)= \{\langle Tx,x\rangle; \ \ x\in\mathscr{H},\ \|x\|=1\}.$$
 This set  is an important   tool which gives many information about  $T$, particularly about its eigenvalues and eigenspaces. The numerical range has a unique nature in numerical analysis and
differential equations. It has many desirable properties, which probably  the most famous of
them is the  Toeplitz-Hausdorff result. It asserts that $ W(T)$ is convex for every
$T\in\mathbb{B}(\mathscr{H})$, see e.g. \cite{GR}.  Moreover, the
basic properties of the numerical range of bounded linear operators on Hilbert
spaces can be found in [5]. We summarize some of them in the following theorem.\\
\textbf{Theorem A.}\cite{GR} \  For $T\in\mathbb{B}(\mathscr{H})$;
\begin{align*}
  &\mathrm{(i)}\ W(\alpha I+\beta T)=\alpha+ \beta W(T), \ \alpha,\beta\in\mathbb{C};\\
    &\mathrm{(ii)}\ W(U^*TU)=W(T), \ \mbox{for every unitary $U\in\mathbb{B}(\mathscr{H})$};\\
     &\mathrm{(iii)}\ sp(T)\subseteq\overline{W(T)},\  \mbox{where $sp(T)$ is the spectrum of $T$}.
\end{align*}

 A related concept is the numerical radius.   The numerical radius  of  $T\in\mathbb{B}(\mathscr{H})$ is defined by
 $$\omega(T)= \sup\{|\lambda|, \ \lambda\in  W(T)\}=\sup\{|\langle Tx,x\rangle|; \ \|x\|=1\}.$$
 Some of basic properties of the numerical radius are listed below.\\
 \textbf{Theorem B.}\ For every $T,S\in\mathbb{B}(\mathscr{H})$
 \begin{align*}
   &\mathrm{(i)}\ \omega(T)=\omega(T^*)\ \mbox{and} \ \omega(U^*TU)=\omega(T) \ \mbox{for every unitary $U\in\mathbb{B}(\mathscr{H})$};\\
   &\mathrm{(ii)}\  \frac{1}{2}\|T\|\leq \omega(T)\leq \|T\| \ \mbox{and} \ \omega(T)=\|T\| \ \mbox{if $T$ is normal};\\
    &\mathrm{(iii)}\  \omega(T\oplus S)=\max\{\omega(T),\omega(S)\};
   \end{align*}
  The numerical radius is also defined for elements of a $C^*$-algebra.  If $\mathscr{A}$ is  a unital $C^*$-algebra, the numerical radius of $A\in\mathscr{A}$ is defined by
 $$\nu(A)=\sup\{|\tau(A)|; \  \tau \ \mbox{is a state on} \ \mathscr{A}   \}.$$
The reader is referred to \cite{ Dr,GR, kit, kit2}  and references therein for more result concerning the numerical radius and the numerical range.

 \section{Matricial Range}
Let $\mathscr{A},\mathscr{B}$ be unital $C^*$-algebras and let $\mathscr{A}_+$ denotes the cone of positive elements of $\mathscr{A}$.  Recall that a mapping $\Phi:\mathscr{A}\to\mathscr{B}$ is called positive, whenever $\Phi(\mathscr{A}_+)\subseteq \mathscr{B}_+$.  Moreover, for $n\in\mathbb{N}$,  $\Phi$ is called $n$-positive if the mapping $\Phi_n:\mathbb{M}_n(\mathscr{A})\to\mathbb{M}_n(\mathscr{B})$ defined by $\Phi_n([A_{ij}])=[\Phi(A_{ij})]$ is positive. If $\Phi:\mathscr{A}\to\mathscr{B}$ is $n$-positive for every $n\in\mathbb{N}$, then $\Phi$ is called completely positive.

For $T\in \mathbb{B}(\mathscr{H})$, assume that $CP_n(T)$ is the set of all unital completely positive linear mappings from $C^*(T)$ to $\mathbb{M}_n$:
$$CP_n(T)=\{\Phi|\ \Phi:C^*(T)\to\mathbb{M}_n\ \mbox{is unital and completely positive} \},$$
in which $C^*(T)$ is the unital $C^*$-algebra generated by $T$. Arveson \cite{Ar} defined the $n$'th matricial range of an operator $T\in \mathbb{B}(\mathscr{H})$ by
 $$W^n(T)=\left\{\Phi(T)|\ \ \Phi\in CP_n(T)\right\}.$$
This is a matrix valued extension of the numerical range, say
$$W^1(T)=\overline{W(T)}.$$   It follows from the definition of $W^n(T)$ that \\
\textbf{Theorem C.}\ If $T\in\mathbb{B}(\mathscr{H})$ and   $n\in\mathbb{N}$, then
\begin{itemize}
\item[(i)] $W^n(T^*)=W^n(T)$;
\item[(ii)] $W^n(U^*TU)=W^n(T)$ for each unitay $U\in\mathbb{B}(\mathscr{H})$;
\item[(iii)] $W^n(\alpha I)=\{\alpha I_n\}$ and $W^n(\alpha T+\beta I)=\alpha W^n(T)+\beta
    I_n$ for all $\alpha,\beta\in\mathbb{C}$.
\end{itemize}

Moreover, as a non-commutative Toeplitz-Hausdorff result, it is known that $W^n(T)$  is $C^*$-convex\cite{Na}. A set $\mathcal{K}\subseteq\mathbb{B}(\mathscr{H})$ is called $C^*$-convex, if  $X_1,\ldots,X_m\in
\mathcal{K}$ and $A_1,\ldots,A_m\in\mathbb{B}(\mathscr{H})$ with $\sum_{j=1}^{m}A^*_jA_j=I$
imply  that $\sum_{j=1}^{m}A^*_jX_jA_j\in \mathcal{K}$.
Indeed, this is a noncommutative generalization of linear convexity.  It is evident that the
$C^*$-convexity of a set implies its  convexity in the usual sense.   But the converse is not
true in general. For more information about $C^*$-convexity see \cite{kian2,LP} and the
references therein.

Matricial ranges are closely connected with $C^*$-convex sets. In fact, the matrix ranges turns
out to be the compact $C^*$-convex sets. However, except in some special cases, it is not
routine to obtain the matricial ranges of an operator.  The reader is referred  to
\cite{Ar,DK,Fa,TW} and the references therein for more information about matricial ranges.

The main purpose of this note is to define an analogues of the numerical radius related to the matricial range. However, we will find   relations between the numerical radius and matricial range of an operator. The tone of the paper is mostly expository.

\section{Matricial Radius}
Similar to the connection of numerical radius and numerical range, it is natural to define the matricial radius of an operator to be the maximum norm of the elements of its matricial range. However,    As pointed out in \cite{Fa}, unlike the numerical radius, the matricial radius is
not interesting. For $T\in\mathbb{B}(\mathscr{H})$ it holds
$$\max \{\|X\|; \ X\in W^n(T)\}=\|T\|.$$

As another candidate for the  matricial radius, we consider the next definition.
 \begin{definition}
    For every operator $T\in\mathbb{B}(\mathscr{H})$  and every positive integer $n$, set
$$\nu^n(T)=\sup\{|\mathrm{Tr} X|; \ X\in W^n(T)\}=\sup\{|\mathrm{Tr}\ \Phi(T)|; \ \Phi\in CP_n(T)\},$$
 \end{definition}
where $\mathrm{Tr}(\cdot)$ denotes the canonical trace. It is easy to see that
\begin{align*}
  &\mathrm{(i)}\ \ \nu^1(T)=\nu(T);\\
  &\mathrm{(ii)}\ \ \nu^n(T^*)= \nu^n(T);\\
  &\mathrm{(iii)}\ \ \nu^n(U^*TU)= \nu^n(T)\ \mbox{for every unitary $U$}.
\end{align*}
Moreover, it can be shown that
$$\nu^n(T)\leq n\|T\|$$
and the equality holds if $T$ is normal.  Although,  $\nu^n$ has some favorite properties, it is not interesting too.
\begin{example}
  Consider
  $$T=\left[\begin{array}{cc}
   0 & 1\\ 0 &0
  \end{array}\right]\in\mathbb{M}_2,$$
  so that $\omega(T)=\frac{1}{2}$ and $\|T\|=1$. Moreover, it is known that \cite{Ar}
  $$W^n(T)=\left\{B\in\mathbb{M}_n \ ; \  \omega(B)\leq\frac{1}{2}\right\}.$$
  Therefore
  $$\nu^2(T)=1=2\ \omega(T).$$
\end{example}
We will show that  the equality $\nu^n(\cdot)=n\ \omega(\cdot)$ holds in general.
 We need some lemmas to continue our work.
\begin{lemma}\label{lm1}\cite{Ar}
  Let $S$ and $T$ be Hilbert space operators (perhaps acting on different spaces) and $S$ is normal. Then the followings are equivalent:\\
  1. $W^n(S)\subseteq W^n(T)$ \\
  2.  $sp(S)$ is contained in the closed numerical range of $T$.
\end{lemma}

The next theorem reveals that $\nu^n$ can not be a proper extension of the numerical radius.
\begin{theorem}\label{th1}
  For every $T\in\mathbb{M}_k$
  $$\omega(T)=\frac{1}{n}\nu^n(T)\qquad (n\in \mathbb{N}).$$
\end{theorem}

\begin{proof}
Assume that $\Phi:C^*(T)\to\mathbb{M}_n$ is a unital completely positive linear mapping. The Arveson’s extension theorem (see for example \cite[Theorem 3.1.5]{Bh}) guarantees the existence of a  unital completely positive linear mapping $ \widetilde{\Phi}:\mathbb{M}_k\to\mathbb{M}_n$, which is an extension of $\Phi$. Moreover, the Stinespring theorem (See \cite[Theorem 3.1.2]{Bh}) yields that $ \widetilde{\Phi}(A)=V^*\pi(A)V$ in which $V:\mathbb{C}^n\to\mathbb{C}^{k^2n}$ and $V^*V=I$ and $\pi:\mathbb{M}_k\to\mathbb{M}_{k^2n}$ is an $*$-homomorphism so that $\pi(A)=A\oplus\cdots\oplus A$. Now, assume that $\{u_1,\cdots,u_n\}$ is an orthonormal system of eigenvectors for $ \widetilde{\Phi}(T)$. Then $Vu_j$ \ $ (j=1,\cdots, n)$ are unit vectors in $\mathbb{C}^{k^2n}$. Therefore
\begin{align*}
|\mathrm{Tr}\ \Phi(T)|=|\mathrm{Tr}\  \widetilde{\Phi}(T)|&=\left|\sum_{j=1}^{n}\langle \widetilde{\Phi}(T)u_j,u_j\rangle \right|\\
&=\left|\sum_{j=1}^{n}\langle V^*\pi(T)Vu_j,u_j\rangle \right|\\
&\leq \sum_{j=1}^{n}|\langle \pi(T)Vu_j,Vu_j\rangle |\\
&\leq \sum_{j=1}^{n} \omega(\pi(T))\\
&=n\ \omega(T\oplus\cdots\oplus T)\\
&=n\ \omega(T),
\end{align*}
where the last inequality follows from $\mathrm{(iii)}$ of Theorem B. Taking supremum over all $\Phi$, we conclude that
\begin{align}\label{q3}
 \nu^n(T)\leq n \ \omega(T).
\end{align}

Furthermore, let $T\in \mathbb{M}_k$. Put $S=\omega(T) I$ so that $S$ is normal and $W^n(S)=\{\omega(T) I_n\}$  by  $\mathrm{(iii)}$ of Theorem C. Moreover,
$sp(S)=\{\omega(T)\}\subseteq \overline{W(T)}$. Lemma \ref{lm1} then implies that $W^n(S)\subseteq W^n(T)$ and so $\nu^n(S)\leq \nu^n(T)$.  Therefore
\begin{align}\label{q4}
n\ \omega (T)=\nu^n(S)\leq \nu^n(T).
\end{align}
The result now follows from \eqref{q3} and \eqref{q4}.
\end{proof}

%%%%%%%%%%%%%%%%%%%%%%%%%%%%%%%%%%%%%%%%%%%%%%%%%%%%%%%%%%%%%%%%%%%%%%%%%%%%%%

The next definition provide another choice for the matricial radius.
\begin{definition}\label{d2}
  For every $T\in\mathbb{B}(\mathscr{H})$
  $$\omega^n(T)=\sup\{\mathrm{Tr}\ |\Phi(T)|; \ \Phi\in CP_n(T)\}=\sup\{\|X\|_1; \ X\in W^n(T)\}.$$
\end{definition}
It is easy to see that
\begin{align*}
  &\mathrm{(i)}\ \omega^1(T)=\nu(T);\\
  &\mathrm{(ii)}\ \omega^n(T^*)= \omega^n(T);\\
  &\mathrm{(iii)}\ \omega^n(U^*TU)= \omega^n(T)\ \mbox{for every unitary $U$}.
\end{align*}
Moreover,  the following desirable property holds for $\omega^n$.
\begin{proposition}\label{p1}
  For every $T\in\mathbb{B}(\mathscr{H})$
  $$\omega^n(T)\leq n\|T\|\qquad (n\in \mathbb{N}). $$
  If $T$ is normal, then equality holds.
\end{proposition}
\begin{proof}
  It is not hard to see that if $\Phi$ is completely positive, then
  \begin{align}\label{q11}
\Phi(T)^*\Phi(T)\leq \|\Phi\|\Phi(T^*T).
  \end{align}
  Noting that $\|\Phi\|=\|\Phi(I)\|=1$ and using the L\"{o}wner--Heinz inequality,  \eqref{q11}  implies that
  \begin{align}\label{q1}
|\Phi(T)|\leq \Phi\left(|T|^2\right)^{1/2}
  \end{align}
   for every unital completely positive linear mapping $\Phi$.   Moreover,
  $$|T|^2\leq \| T\|^2 I.$$
  Now assume that $\Phi:C^*(T)\to\mathbb{M}_n$ is a unital completely positive linear mapping. It follows from the last inequality that
  \begin{align}\label{q2}
    \Phi\left(|T|^2\right)^{1/2}\leq \| T\|\ I_n.
  \end{align}
  From \eqref{q1} and \eqref{q2} we get
  $$|\Phi(T)|\leq \| T\|\ I_n$$
  and so
  $$\mathrm{Tr}\ |\Phi(T)|\leq n\ \| T\|.$$
  This concludes the inequality $\omega^n(T)\leq n\|T\|$ for every $T\in\mathbb{B}(\mathscr{H})$.

Now assume that $T$ is normal. Then the Gelfand mapping $\Gamma:C^*(T)\to C(sp(T))$ is an isometric $*$-isomorphism, where $C(sp(T))$ is the $C^*$-algebra of all continuous functions on $sp(T)$. Consider two facts:\\
1. Every positive linear mapping $\Phi:C(\Omega)\to\mathscr{A}$ is completely positive for each arbitrary $C^*$-algebra $\mathscr{A}$\cite{Bh};\\
2. The composition of every two completely positive linear mapping is completely positive too.

Every positive linear mapping  $\Phi:C^*(T)\to\mathbb{M}_n$  can be written as $\Phi= \Psi o \Gamma$, where $\Psi=\Phi o \Gamma^{-1}:C(sp(T))\to \mathbb{M}_n$. Therefore, every positive linear mapping $\Phi:C^*(T)\to\mathbb{M}_n$ is completely positive.

Now let $x\in\mathscr{H}$ be a unit vector. The linear mapping $\Phi_x:C^*(T)\to\mathbb{M}_n$ defined by $\Phi(Z)=\langle Zx,x\rangle I_n$ is   positive and so is completely positive. Therefore,
$$\omega^n(T)\geq \mathrm{Tr} \ |\Phi_x(T)|= \mathrm{Tr} \ |\langle Tx,x\rangle I_n |=n |\langle Tx,x\rangle|,$$
whence
$$\omega^n(T)\geq n \ \omega(T)=n\|T\|.$$
\end{proof}
%%%%%%%%%%%%%%%%%%%%%%%%%%%%%%%%%%%%%%%%%%%%%%%%%%%%%%%%%%%%%%%%%%%%%%%%%%%%%%%
Proposition \ref{p1} gives an extension of $\mathrm{(ii)}$ of Theorem B.   Note that there exists other   norms on $\mathbb{M}_n$ which can be used in Definition \ref{d2} rather than $\|\cdot\|_1$. Typical norms on  $\mathbb{M}_n$ are
$$\|A\|_p=\mathrm{Tr}\left(|A|^p\right)^{1/p}\quad \mbox{and}\quad \|A\|=\lim_{p\to\infty}\|A\|_p\qquad (A\in\mathbb{M}_n)$$
in which $\|A\| $ is the operator norm. Except when $p=1$, Proposition \ref{p1} does not hold in general.
To see this, consider  the unilateral shift operator defined on a separable Hilbert space  by $Te_j=e_{j+1}$ ($j\geq1$).  It is known that \cite{Fa}
\begin{align*}
W^n(T)=\{B\in\mathbb{M}_n\ ; \ B^*B\leq I_n\}.
  \end{align*}
Therefore,
$$\omega^n(T)=n=n\|T\|.$$
Considering the $p$-norm ($p\neq 1$) in Definition \ref{d2} concludes
$$\sup\{\|X\|_p; \ X\in W^n(T)\}=\sqrt[p]{n}\neq n\ \|T\|.$$

Unfortunately, Definition \ref{d2} can not be a proper extension of  the numerical radius too.

\begin{theorem}
For every $T\in\mathbb{M}_k$
  $$\omega(T)=\frac{1}{n}\omega^n(T)\qquad (n\in \mathbb{N}).$$
\end{theorem}
\begin{proof}
It is known that (see \cite[Theorem 3.7]{kian})
$$\|A\|_1\leq n\ \omega(A)\qquad (A\in\mathbb{M}_n).$$
Moreover, for $T\in\mathbb{M}_k$, it is known that $W^m(W^n(T))\subseteq W^m(T)$ for all $m,n\in\mathbb{N}$ \cite{Fa}, i.e., if $A\in W^n(T)$, then $W^m(A)\subseteq W^m(T)$. Therefore, $\omega(A)\leq \omega(T)$. It follows that
$$\|X\|_1\leq n \ \omega(X)\leq n \ \omega(T)\qquad (X\in W^n(T)),$$
whence
$$\omega^n(T)\leq n \ \omega(T).$$
 Furthermore,  applying  an argument as in the last part of the proof of Theorem \ref{th1} shows that $$n \ \omega(T)\leq \omega^n(T).$$
This completes the proof.
\end{proof}
\begin{example}
  Assume that $r>0$ and
   $$T=\left[\begin{array}{cc}
   0 & r\\ 0 &0
  \end{array}\right]\in\mathbb{M}_2,$$
  so that $\omega(T)=\frac{r}{2}$ and $\|T\|=r$ and
  $$W^n(T)=\left\{B\in\mathbb{M}_n \ ; \  \omega(B)\leq\frac{r}{2}\right\}.$$
We have
\begin{align}\label{nr}
\omega^n(T)=\sup\{\|X\|_1; \ X\in W^n(T)\}\leq n \sup\{\omega(X); \ X\in W^n(T)\}\leq \frac{nr}{2}=n\omega(T).
\end{align}
  Moreover, put $Y=\frac{r}{2}I_n\in W^n(T)$ and then
  $$\omega^n(T)=\sup\{\|X\|_1; \ X\in W^n(T)\}\geq \|Y\|_1=\frac{nr}{2}=n\omega(T),$$
  whence, $$\omega^n(T)=n\omega(T).$$
\end{example}
\begin{remark}
  First, we can not find a suitable extension of the numerical radius based on the matricial range. So, we would like to pose this question that is there such an extension. Second, we obtain some relations of the numerical radius of an operator  with its matricial range. In particular,
  \begin{align*}
    \omega(T)=&=\frac{1}{n}\sup\left\{|\mathrm{Tr}\ X|;\ X\in W^n(T)\right\}\\
&=\frac{1}{n}\sup\left\{\|X\|_1;\ X\in W^n(T)\right\}\\
&= \sup\left\{ \omega(X);\ X\in W^n(T)\right\}.
\end{align*}
The last equality follows from \eqref{nr}.
\end{remark}

%%%%%%%%%%%%%%%%%%%%%%%%%%%%%%%%%%%%%%%%%%%%%%%%%%%%%%%%%%%%%%%%%%%%%%%%%%%%%%%%%%%

%%%=================================================================================
\bibliographystyle{amsplain}

\end{document}